\documentclass{amsart}
\usepackage{amsmath,amssymb,amsthm}
\usepackage{color} 
\definecolor{darkblue}{rgb}{0,0,0.6}
\usepackage[breaklinks, pdftex, ocgcolorlinks, citecolor=darkblue, filecolor=darkblue, linkcolor=darkblue, urlcolor=darkblue]{hyperref}
\usepackage[capitalize,nameinlink]{cleveref}
\usepackage{enumerate}
\usepackage[normalem]{ulem}
\usepackage[all]{xy}
\usepackage{tikz}
\usepackage{tikz-cd}
	\tikzcdset{row sep/normal=1.3em}
	\tikzcdset{column sep/normal=1.3em}

\numberwithin{equation}{section}

\makeatletter
\def\namedlabel#1#2{\begingroup
	#2%
	\def\@currentlabel{#2}%
	\phantomsection\label{#1}\endgroup
}
\makeatother

\newtheorem{thm}[equation]{Theorem}
\newtheorem{lem}[equation]{Lemma}

\newtheorem{cor}[equation]{Corollary}
\theoremstyle{definition}
\newtheorem{defi}[equation]{Definition}
\newtheorem{rem}[equation]{Remark}

\DeclareMathOperator{\id}{id}
\newcommand{\Z}{\mathbb Z}
\newcommand{\CP}{\mathbb{CP}}

\DeclareMathOperator{\ks}{ks}

\newcommand{\wt}{\widetilde}

\renewcommand{\L}{\mathrm{L}}

\renewcommand{\SS}{\mathbb{S}}
\newcommand{\lto}{\longrightarrow}

\title{Topological 4-manifolds with 4-dimensional fundamental group}

\author{Daniel Kasprowski}
\address{Rheinische Friedrich-Wilhelms-Universit\"{a}t Bonn, Mathematisches Institut,\newline \indent Endenicher Allee 60, 53115 Bonn, Germany}
\email{\href{mailto:kasprowski@uni-bonn.de}{kasprowski@uni-bonn.de}}

\author{Markus Land}
\address{Department of Mathematical Sciences, University of Copenhagen, \newline \indent 2100 Copenhagen, Denmark}
\email{\href{mailto:markus.land@math.ku.dk}{markus.land@math.ku.dk}}

\def\subjclassname{\textup{2020} Mathematics Subject Classification}
\expandafter\let\csname subjclassname@1991\endcsname=\subjclassname
\expandafter\let\csname subjclassname@2000\endcsname=\subjclassname

\date{\today}

\begin{document}

\begin{abstract}
	Let $\pi$ be a group satisfying the Farrell--Jones conjecture and assume that $B\pi$ is a 4-dimensional Poincar\'e duality space. We consider topological, closed, connected manifolds with fundamental group $\pi$ whose canonical map to $B\pi$ has degree 1, and show that two such manifolds are $s$-cobordant if and only if their equivariant intersection forms are isometric and they have the same Kirby--Siebenmann invariant. If $\pi$ is good in the sense of Freedman, it follows that two such manifolds are homeomorphic if and only if they are homotopy equivalent and have the same Kirby--Siebenmann invariant. This shows rigidity in many cases that lie between aspherical 4-manifolds, where rigidity is expected by Borel's conjecture, and simply connected manifolds where rigidity is a consequence of Freedman's classification results.
\end{abstract}
\maketitle

\section{Introduction}
The classification of closed 4-manifolds remains one of the most exciting open problems in low-dimensional topology. In the topological category, classification is only known for closed 4-manifolds with trivial~\cite{Freedman82}, cyclic~\cite{Freedman-Quinn, surgeryandduality, hambleton-kreck-93-III} or Baumslag-Solitar~\cite{HKT} fundamental group. Here we give a partial answer in the case where the fundamental group $\pi$ satisfies the Farrell--Jones conjecture and is such that $B\pi$ is a 4-dimensional Poincar\'e duality space. For the definition of a Poincar\'e duality space we refer to \cite[Chapter~1]{Wall67}. The condition on the Farrell--Jones conjecture is, by extensive work of Bartels--Reich--L\"uck and collaborators, by now a rather weak assumption \cite[Theorem 14.1]{LueckFJC}. For example, one can consider $\pi$ being $\Z^4$ or an extension of a surface group by a surface group \cite[14.23]{LueckFJC}.
In this paper, all 4-manifolds are assumed to be topological, closed, and connected unless specified otherwise. For the formulation of our main theorem, we recall that a manifold is \emph{almost spin} if its universal cover is spin. 
\begin{thm}
	\label{thm:main}
	Let $\pi$ be the fundamental group of an aspherical 4-dimensional Poincar\'e duality space with orientation character $w$ and assume that $\pi$ satisfies the Farrell--Jones conjectures.
	Let $M$ and $N$ be $4$-manifolds with fundamental group $\pi$ and orientation character $w$ such that the classifying maps $M\to B\pi$ and $N\to B\pi$ have degree $\pm 1$, i.e.\ that the induced maps $H_4(M;\Z^w)\to H_4(B\pi;\Z^w)$ and $H_4(N;\Z^w)\to H_4(B\pi;\Z^w)$ are isomorphisms. 
	
	Then $M$ and $N$ are $s$-cobordant if and only if their equivariant intersection forms are isometric and they have the same Kirby--Siebenmann invariant. 
	In the case that $M$ and $N$ are almost spin, they are $s$-cobordant if and only if their equivariant intersection forms are isometric.
\end{thm}
Here $\Z^{w}$ refers to the $\pi$-module obtained from the sign representation of $C_2=\{\pm 1\}$ on $\Z$ by restricting along the orientation character $w\colon \pi \to C_2$.

\begin{rem}\label{remark3}
We note that the classifying map $M\to B\pi$ has degree $\pm 1$ if and only if $\pi_2(M)$ is a projective $\Z\pi$-module, \cite[Theorem~6]{hillman}. Hence it follows directly from \cite[Theorem~11]{hillman} that under the assumptions of \cref{thm:main} $M$ and $N$ are homotopy equivalent, and we will use this in the proof.
Moreover we note that $M$ is almost spin if and only if the $\Z$-valued intersection form on $\pi_2(M) \cong H_2(\widetilde{M};\Z)$ is even.
\end{rem}

If we furthermore assume that $\pi$ is good in the sense of Freedman, then the topological $s$-cobordism theorem holds in dimension four and we obtain the following corollary. Note that the class of good groups includes all solvable groups~\cite{FT, KQ}, which in particular includes the fundamental groups of torus bundles over tori.
\begin{cor}\label{cor1}
As in \cref{thm:main}, let $\pi$ be the fundamental group of a 4-dimensional Poincar\'e duality space with orientation character $w$ and assume that $\pi$ satisfies the Farrell--Jones conjecture. Let $M$ and $N$ be $4$-manifolds with fundamental group $\pi$ and orientation character $w$ such that the classifying maps $M\to B\pi$ and $N\to B\pi$ have degree $\pm 1$. Assume in addition that $\pi$ is good in the sense of Freedman. 

Then $M$ and $N$ are homeomorphic if and only if their equivariant intersection forms are isometric and they have the same Kirby--Siebenmann invariant. If $M$ and $N$ are almost spin, then they are homeomorphic if and only if their equivariant intersection forms are isometric.
\end{cor}
The following corollary is an immediate consequence. The first part was previously obtained in \cite[Theorem 0.11]{KL}. Using this, the second part also follows from \cite[Theorem~1.2]{CH}.
\begin{cor}\label{cor2}
Let $X$ be a topological 4-manifold which is aspherical and whose fundamental group $\pi$ is a good Farrell--Jones group, e.g.\ a solvable group, and let $L$ be a simply connected 4-manifold. Let $M$ be a $4$-manifold with fundamental group $\pi$ and orientation character $w_1(M)=w_1(X)$.
\begin{enumerate}
\item If $M$ is homotopy equivalent to $X \# L$, then it is homeomorphic to $X \# L$ or $X\# \star L$.\footnote{Recall that a star partner $\star L$ of the simply connected 4-manifold $L$ is a closed $4$-manifold which is homotopy equivalent to $L$ and has opposite Kirby--Siebenmann invariant. Such a star partner exists if and only if $L$ is not spin \cite[Section~10.4]{Freedman-Quinn}, and in case $L$ is spin, we simply set $\star L = L$.}
\item If the equivariant intersection form of $M$ is induced from an integral form $\lambda$, then $M$ is homeomorphic to $X\#K$ for a simply connected $4$-manifold $K$ whose intersection form is $\lambda$.

\end{enumerate}
\end{cor}
\begin{proof}
We first prove Assertion (2). By Freedman's classification of simply connected 4-manifolds \cite[Theorem~1.5]{Freedman82}, there is a simply connected $4$-manifold $K$ with intersection form $\lambda$, such that $\lambda_M \cong \lambda\otimes_\Z \Z\pi$. We wish to apply \cref{cor1} to $M$ and $X \# K$. By \cref{remark3}, $M$ is almost spin if and only if $X \# K$ is (which is the case if and only if $K$ is spin), and in this case \cref{cor1} applies immediately, since the collapse map $X\# L\to X$ has degree $\pm 1$.
If $M$ is totally non-spin, then $K$ is not spin. Hence there exists a star partner $\star K$ with
opposite Kirby--Siebenmann invariant. Then either $X\# K$ or $X \# \star K$ has the same Kirby--Siebenmann as $M$, so we may apply \cref{cor1}.

The same argument gives Assertion (1), noting that if $M$ is homotopy equivalent to $X\# L$, then its intersection form is induced from the (integral) intersection form of $L$.
\end{proof}

We interpret this as saying that manifolds of the kind $X \# L$ as in the corollary are \emph{weakly rigid}, i.e.\ that any two such manifolds which are homotopy equivalent (equivalently whose intersection forms are isomorphic) are already homeomorphic. A variant of such weak rigidity properties was studied by \cite{KL} and \cite{khan}. 
This builds a bridge between 
the rigidity phenomena envisioned by Borel for aspherical manifolds and the rigidity present in simply connected topological 4-manifolds by Freedman's results.

Besides \cite[Theorem~11]{hillman} the main tool is the surgery sequence. Even though it is not exact in general, it suffices for proving \cref{thm:main}. In a previous version we deduced \cref{thm:main} from Kreck's modified surgery and we thank an anonymous referee for pointing out the current approach which shortens the proof considerably.

\subsection*{Acknowledgements}
We thank Diarmuid Crowley, Mark Powell and anonymous referees for helpful comments on a previous version of this paper.
The first author was funded by the Deutsche Forschungsgemeinschaft (DFG, German Research Foundation) under Germany's Excellence Strategy - GZ 2047/1, Projekt-ID 390685813. The second author was supported by the SFB 1085 ``Higher Invariants'' in Regensburg, by the DFG through a research fellowship and by the Danish National Research Foundation through the Centre for Symmetry and Deformation (DNRF92) and the Copenhagen Centre for Topology and Geometry (DNRF151). We thank the MATRIX centre for hospitality during the workshop ``Topology of Manifolds: Interactions between high and low dimensions''.

\section{L-theoretic preliminaries}
We will denote by $\L^q(R)$ the quadratic (free) L-theory spectrum of a ring and let $\L^q_n(R) := \pi_n(\L^q(R))$ denote the $n$th quadratic L-group of $R$. Following \cite{KLPT15}, we set $\L = \L^q(\Z)$ and let $\L\langle 1 \rangle$ be its connected cover.

Since we work in the possibly non-orientable situation, the surgery obstruction groups are $w$-twisted L-groups, i.e.\ the quadratic L-groups $\L^q_*(\Z G,w)$ of the ring $\Z G$ whose involution is induced by $g \mapsto w(g)\cdot g^{-1}$, for some orientation character $w\colon G \to C_2$. In the geometric situation the group $G$ is the fundamental group of a
manifold $M$ and $w = w_1(M)$ is its orientation character. The surgery obstruction map is, as in the oriented case, equivalent to a map induced by the assembly map in $\L$-theory, but now taking the orientation character $w$ into account, see \cite[Appendix A]{Ranicki-blue-book}. More precisely, 
the surgery obstruction map is isomorphic to the composite
\[ \L\langle 1\rangle_*^w(BG) \lto \L^w_*(BG) \lto \L^q_*(\Z G,w) \]
where $\L\langle1\rangle^w_*(BG)$ and $\L^w_*(BG)$ denote the 
$w$-twisted $\L\langle 1 \rangle$- and $\L$-homology of $BG$, respectively. The first map in the composite is induced by the canonical map $\L\langle1\rangle \to \L$ and the second map is the assembly map. Here, the twist comes from pulling back a particular $C_2$-action on $\L$ along the orientation character $w \colon G \to C_2$, see \cref{remark:C2-actions} for more information on this $C_2$-action. 

\begin{rem}\label{remark:C2-actions}
The $C_2$-action on $\L$ is in fact given by $\SS^{\sigma-1}\otimes \L$, where the $C_2$-action on $\SS^{\sigma-1}$ will be recalled below and the action on $\L$ is trivial\footnote{We follow the convention of writing $\otimes$ for the symmetric monoidal structure of spectra which is often also denoted $\wedge$.}. 
Discarding the $C_2$-action, we have $\SS^{\sigma-1}\otimes \L = \L$ since $\SS^{\sigma-1} = \SS$ once we discard the $C_2$-action. The above identification of the $C_2$-action becomes relevant when one identifies normal invariants of non-orientable manifolds with twisted $\L$-homology, as needed when showing that the Borel conjecture follows from the Farrell--Jones conjectures also for non-orientable manifolds. In addition, it can be used to determine differentials in the equivariant Atiyah--Hirzebruch spectral sequence for twisted $\L$-homology, see \cref{remark:equivariant-k-invariant}. The rest of this remark is intended to give some context on $C_2$-actions on $\L$. It will not be used in this paper and may therefore safely be skipped at first reading.

So let us explain a bit more about $C_2$-actions on $\L$. Denote by $S^\sigma$ the one-point compactification of the sign representation of $C_2$ on $\mathbb{R}$, in other words, the space $S^1$ with $C_2$-action given by complex conjugation. Let $\SS^\sigma$ be its suspension spectrum, and $\SS^{\sigma-1}$ its one-fold desuspension. This is a spectrum with an action of $C_2$, whose underlying spectrum is just the usual sphere spectrum $\SS$ and where the generator of $C_2$ acts by multiplication by $-1$. While in algebra, a $C_2$-action is determined by this property, this is not the case in topology. For instance, denote by $n\sigma$ the one-point compactification of an $n$-fold direct sum of the sign representation $\sigma$, and consider $\SS^{n\sigma-n}$. This is also a spectrum with $C_2$-action whose underlying spectrum is $\SS$ and where the generator acts by $(-1)^n$, and thus by a sign if $n$ is odd. However, as spectra with $C_2$-action, all of these turn out to be pairwise inequivalent. Indeed, there is a functor, the $C_2$-Tate construction, which sends $\SS^{n\sigma-n}$ to $\SS^{-n}$. Unlike in algebra, one therefore has to be more careful in general when talking about a $C_2$-action ``by a sign'' on a spectrum $X$, as for odd $n$, any $\SS^{n\sigma-n}\otimes X$ is a reasonable such object. We remark here that for specific $X$, the family of spectra with $C_2$-action $\SS^{n\sigma-n}\otimes X$ simplifies: if $M$ is an $\mathrm{MU}$-module spectrum\footnote{That is, $M$ is equipped with a homotopy unital and homotopy associative multiplication map $\mathrm{MU} \otimes M \to M$, where $\mathrm{MU}$ denotes the complex cobordism spectrum.}, it turns out that for any complex representation $V$ of $C_2$, there is an equivalence $S^V \otimes M \simeq S^{|V|}\otimes M$ of spectra with $C_2$-action, where the latter is given the trivial action. Here, $|V|$ denotes the real dimension of $V$, and the equivalence is a form of a Thom isomorphism for the complex representation $V$. Since $2\sigma$ is a complex representation, we therefore find that $\SS^{2n\sigma-2n} \otimes M$ is equivalent to $M$ with the trivial action, and that $\SS^{(2n+1)\sigma - 2n-1} \otimes M$ is equivalent to $\SS^{\sigma-1}\otimes M$. The spectrum $\L$ is a module spectrum over $\L^s(\Z)$ which in turn comes with ring maps $\mathrm{MU}\to \mathrm{MSO} \to \L^s(\Z)$. Therefore, $\L$ is also a module spectrum over $\mathrm{MU}$ and thus there is a well-defined ``sign-action'' of $C_2$ on $\L$.
\end{rem}

Following \cite{HKT}, we consider the following conditions on a group $\pi$ equipped with an orientation character $w$:
\begin{description}
	\item[\namedlabel{cond:W}{(W)}] The Whitehead group $\mathrm{Wh}(\pi)$ vanishes.
	\item[\namedlabel{cond:A1}{(A1)}] The assembly map $\L\langle 1 \rangle_4^w(B\pi) \to \L_4^q(\Z\pi,w)$ is injective.
	\item[\namedlabel{cond:A2}{(A2)}]\label{cond:A2} The assembly map $\L\langle 1 \rangle_5^w(B\pi) \to \L_5^q(\Z\pi,w)$ is surjective.
\end{description}
\begin{defi}[{\cite{HKT}}]
	A group is said to satisfy property 
	\begin{description}
		\item[\namedlabel{cond:WAA}{(W-AA)}] if it satisfies the three conditions above. 
	\end{description}
\end{defi}

\begin{lem}\label{lemma-WAA}
	Let $\pi$ be a group satisfying the Farrell--Jones conjectures and admitting a CW-model of $B\pi$ of dimension at most $4$. Then $\pi$ has property \ref{cond:WAA}.
\end{lem}
\begin{proof}
	First note that the existence of a finite-dimensional model for $B\pi$ implies that $\pi$ is torsion-free. In this case, the $K$-theoretic Farrell--Jones conjecture implies property \ref{cond:W} as well as the vanishing of $\wt K_{n}(\Z\pi)$ for $n\leq 0$, see e.g.\ \cite[7.2.2]{Lueck-HH}. Indeed, the $K$-theoretic Farrell--Jones conjecture in the case of a torsion free group asserts that the assembly map
\[ K(\Z)_n(B\pi) \lto K_n(\Z\pi) \]
is an isomorphism for all $n\in \Z$. Moreover, 
in non-positive degrees the cokernels of these maps are the groups $\wt{K}_n(\Z\pi)$, and the cokernel in degree 1 is $\mathrm{Wh}(\pi)$.
Using the long exact Rothenberg sequences \cite[\S 16]{Ranicki-lower}, the $\L$-theoretic Farrell--Jones conjecture implies that the assembly map
	\[\L_n^w(B\pi)\lto \L_n^q(\Z\pi,w)\]
	is an isomorphism for all $n\in\Z$. Now, consider the following long exact sequence 
	\[ \dots \to (\tau_{\leq 0}\L)_{n+1}^w(B\pi) \to \L\langle 1 \rangle_n^w(B\pi) \to \L_n^w(B\pi) \to (\tau_{\leq 0}\L)_n^w(B\pi) \to \dots\]
	where $\tau_{\leq 0}\L$ denotes the cofibre of the canonical map $\L\langle 1 \rangle \to \L$. We note that the $n$th homotopy group $(\tau_{\leq 0}\L)_n$ vanishes for $n >0$. As in the untwisted case, this long exact sequence is induced by a fibre sequence of spectra parametrised over $B\pi$
	\[ \L\langle 1\rangle^w \lto \L^w \lto \tau_{\leq 0} \L^w\]
	where the parametrisation is via the orientation character $w$.
	Hence for showing properties \ref{cond:A1} and \ref{cond:A2} it suffices to show that $(\tau_{\leq 0}\L)_{5}^w(B\pi) = 0$. This follows from the twisted Atiyah--Hirzebruch spectral sequence calculating the $w$-twisted $\tau_{\leq 0}\L$-homology of $B\pi$. Indeed, its $E^2$-page is given by $E^2_{p,q} = H_p(B\pi;(\tau_{\leq 0}\L)^w_q)$, which vanishes for $p+q=5$ as $\pi$ has a 4-dimensional model for $B\pi$ by assumption. 
\end{proof}

Property \ref{cond:WAA} is what we will actually use in the proof of \cref{thm:main}.
\begin{rem}
	\label{rem:condition}
	We also note the well-known fact that condition \ref{cond:W} implies that the comparison map from simple L-groups of $\Z\pi$ to the (free) L-groups of $\Z\pi$ are isomorphisms, as the canonical comparison map sits inside the long exact Rothenberg sequence whose third term is $\widehat{H}^*(C_2;\mathrm{Wh}(\pi))$ \cite{Shaneson}; see also \cite[\S 1.10]{Ranickiyellow}. Likewise, under condition \ref{cond:W}, the simple structure set $\mathcal{S}^s(M)$ agrees, via the canonical map, with the non-simple version $\mathcal{S}(M)$. We may therefore always work with the non-simple versions in what follows.
\end{rem}

\section{Proof of \cref{thm:main}}

We will need the following theorem, which follows immediately from \cite[Theorem~4 and its addendum]{kirby-taylor}. For convenience of the reader we briefly explain the proof. We follow the idea from the proof of \cite[Theorem~2.6]{HKT}. 
\begin{thm}
	\label{thm:KT}
	Assume $\pi$ satisfies property \ref{cond:WAA}. Let $M$ and $N$ be as in \cref{thm:main}. Let $f\colon N\to M$ be a homotopy equivalence with trivial normal invariant. Then $N$ and $M$ are $s$-cobordant.
\end{thm}
\begin{proof}
	The vanishing of the Whitehead group implies that every homotopy equivalence is simple and, as explained in \cref{rem:condition}, that we do not have to distinguish between simple L-theory and free L-theory. By assumption, $f$ and $\id_M$ are normally bordant, so applying surgery in the interior of the normal bordism we obtain a normal bordism $W$ between $f$ and $\id_M$ with a 2-equivalence to $M$ and surgery obstruction in $\L_5^q(\Z\pi,w)$. Let $\mathcal{N}(M\times I, M\times\{0,1\})$ be the set of degree one normal maps to $M\times I$ that are the identity on both boundary components. Given an element $V$ from $\mathcal{N}(M\times I, M\times\{0,1\})$, we can glue it along a boundary component of the bordism $W$. 
The surgery obstruction is additive under stacking cobordisms since the surgery kernel of the stacked cobordism is the orthogonal sum of the surgery kernels of the individual cobordisms. Therefore, glueing $V$ to $W$ changes the surgery obstruction by the image of $V$ in $\L_5^q(\Z\pi,w)$. Hence if the map $\mathcal{N}(M\times I, M\times\{0,1\})\to \L_5^q(\Z\pi,w)$ is surjective, the surgery obstruction can be assumed to be zero, and so there is an $s$-cobordism between $M$ and $N$.
	
	It remains to show that $\mathcal{N}(M\times I, M\times\{0,1\})\to \L_5^q(\Z\pi,w)$ is surjective. Under the identification of $\mathcal{N}(M\times I, M\times\{0,1\})$ with $\L\langle1\rangle_5^w(M)$, the map to $\L_5^q(\Z\pi,w)$ agrees with the assembly map $\L\langle1\rangle_5^w(M)\to \L_5^q(\Z\pi,w)$ which factors through $\L\langle 1\rangle_5^w(B\pi)$. The map $\L\langle 1\rangle_5^w(B\pi)\to \L_5^q(\Z\pi,w)$ is surjective by property \ref{cond:A2}.
	Hence it remains to show that $\L\langle1\rangle_5^w(M)\to \L\langle1\rangle_5^w(B\pi)$ is surjective. Indeed, it is an isomorphism. To see this, we consider the twisted Atiyah--Hirzebruch spectral sequences for $M$ and $B\pi$ and obtain the following commutative diagram:
\[\begin{tikzcd}[column sep=tiny]
	H_4(M;\Z/2) \ar[r] \ar[d,"\cong"] & H_1(M;\Z^{w}) \ar[r] \ar[d,"\cong"] & \L\langle1 \rangle_5^w(M) \ar[r] \ar[d] & H_3(M;\Z/2) \ar[r] \ar[d,"\cong"] & H_0(M;\Z^{w}) \ar[d,"\cong"] \\
	H_4(B\pi;\Z/2) \ar[r] & H_1(B\pi;\Z^{w}) \ar[r] & \L\langle 1\rangle_5^w(B\pi) \ar[r] & H_3(B\pi;\Z/2) \ar[r] & H_0(B\pi;\Z^{w})
\end{tikzcd}\]
Since $M \to B\pi$ is degree $1$ and a $\pi_1$-isomorphisms, we find that all vertical maps except for the middle one are isomorphisms. It follows that also the middle one is an isomorphism.
\end{proof}

\begin{rem}\label{remark:equivariant-k-invariant}
In \cite{kirby-taylor} it is observed that $\L\langle 1\rangle^w_5(X) \cong H_1(X;\Z^{w})\oplus H_3(X;\Z/2)$, but we shall not make use of this. One way to see this is to use that the $C_2$-action on $\L$ used for the twisted homology is given by $\SS^{\sigma-1} \otimes \L$ as described in \cref{remark:C2-actions}. Then it follows that $\tau_{[1,4]}\L$ is, as a $C_2$-object, also given by $\SS^{\sigma-1} \otimes \tau_{[1,4]} \L$. From the fact that $\tau_{[1,4]}\L$ unequivariantly splits into the product $\Sigma^2H\Z/2 \oplus \Sigma^4H\Z$, one therefore finds that there is an equivariant splitting of $\tau_{[1,4]}\L$ into $\Sigma^2H\Z/2 \oplus \Sigma^4 H\Z^{-}$, where $\Z^{-}$ denotes the sign representation of $C_2$ acting on $\Z$.
\end{rem}

We need two more lemmas to prove \cref{thm:main}.

\begin{lem}\label{lem3}
	If $f\colon N \to M$ is a homotopy equivalence such that $N(f)$ is represented by an immersed sphere with trivial $w_2$, then there exists a homotopy equivalence $f'\colon N\to M$ with trivial normal invariant.
\end{lem}
\begin{proof}
	By \cite{cochran-habegger}, more precisely by the proof of \cite[Theorem 5.1]{cochran-habegger}, for every element $\iota$ of $H_2(M;\Z/2)$ which is represented by an immersed $S^2$ with $w_2(\iota)=0$, there exists a self homotopy equivalence $\varphi$ of $M$ with $N(\varphi) = \iota$; see also \cite[Theorem 18]{kirby-taylor}. Explicitly, this homotopy equivalence is given by the Novikov pinch
\[ M \stackrel{\mathrm{pinch}}{\lto} M\vee S^4 \stackrel{\mathrm{id}\vee \eta^2}{\lto} M\vee S^2 \stackrel{\mathrm{id}+ \iota}{\lto} M.\]
In particular, there exists a self homotopy equivalence $\varphi$ of $M$ with $N(\varphi) = N(f)$ so that $\varphi$ is normally bordant to $f$. Hence, $\varphi^{-1} \circ f$ is a homotopy equivalence from $N$ to $M$ which is normally bordant to $\id_M$ as needed. \end{proof}

\begin{lem}\label{lem4}
Let $M$ be totally non-spin and $M'$ a manifold homotopy equivalent but not $s$-cobordant to $M$. Then $\ks(M) \neq \ks(M')$.
\end{lem}
\begin{proof}
Let $f \colon M' \to M$ be a homotopy equivalence such that $M'$ is not $s$-cobordant to $M$. \cref{lem3} implies that its normal invariant $N(f)$ is represented by an immersed sphere $\alpha$ with non-vanishing $w_2$. Consider a homotopy equivalence $g \colon \star \CP^2 \to \CP^2$, where $\star \CP^2$ is the Chern manifold. Then the normal invariant $N(g)\in H_2(\CP^2;\Z/2)\cong \Z/2$ is non-trivial and thus represented by $\CP^1\subseteq \CP^2$. Considering the the homotopy equivalence $M'\#\star\CP^2\xrightarrow{f\#g}M\#\CP^2$, we find that $N(f\# g)$ is represented by $\alpha + \CP^1 \in H_2(M'\# \CP^2;\Z/2)$. Using the result from \cite{cochran-habegger} again and the fact that $w_2(\alpha + \CP^1) = 0$, it follows that $M'\#\star\CP^2$ is s-cobordant to $M\#\CP^2$ and thus $\ks(M')+1=\ks(M'\#\star\CP^2)=\ks(M\#\CP^2)=\ks(M)$ as claimed.
\end{proof}

\begin{proof}[Proof of \cref{thm:main}]
	As mentioned in \cref{remark3}, the classifying map $M\to B\pi$ has degree $\pm 1$ if and only if $\pi_2(M)$ is projective, \cite[Theorem~6]{hillman}. Hence it follows directly from \cite[Theorem~11]{hillman} that under the assumptions of \cref{thm:main} there is a homotopy equivalence $f\colon N\to M$. We view $f$ as an element in $\mathcal{S}(M)$. Using \ref{cond:A1}, the normal invariant $N(f)$ of $f$ lies in the kernel of $\L\langle 1 \rangle_4^w(M)\to \L\langle 1 \rangle_4^w(B\pi)$. Using a similar argument as at the end of the proof of \cref{thm:KT} with the twisted Atiyah--Hirzebruch spectral sequence, we find that the kernel of $\L\langle 1 \rangle_4^w(M)\to \L\langle 1 \rangle_4^w(B\pi)$ is isomorphic to the kernel of $H_2(M;\Z/2)\to H_2(B\pi;\Z/2)$. The Serre spectral sequence for the fibration $\wt M \to M \to B\pi$ shows that $H_2(\wt M;\Z/2)\cong H_2(\wt M;\Z)\otimes_\Z\Z/2\cong \pi_2(\wt M)\otimes_\Z \Z/2$ surjects onto the kernel of $H_2(M;\Z/2)\to H_2(B\pi;\Z/2)$. Thus, every element in this kernel can be represented by an immersed sphere. 
	
	If $M$ and $N$ in \cref{thm:main} are almost spin, all immersed spheres in $M$ have trivial $w_2$ by assumption.  In this case, \cref{lem3} together with \cref{thm:KT} proves \cref{thm:main}. Note that no appeal to Kirby--Siebenmann invariants was necessary up to this point.
	
	If $M$ is almost spin, \cref{thm:main} follows from \cref{lem4} since we assumed that $N$ has the same Kirby--Siebenmann invariant as $M$.
\end{proof}

\begin{rem}
If $M$ is totally non-spin and $\pi=\pi_1(M)$ is good, one can give a different argument for \cref{lem4} making use of 
the existence of a star partner $\star M$ of $M$, i.e.\ a manifold homotopy equivalent to $M$ but with opposite Kirby--Siebenmann invariant. Indeed, if such a star partner exists, it follows immediately from the observation that the orbit space $\mathcal{S}(M)$ under the action of the self homotopy equivalences of $M$ consists of 2 elements, that they must be represented by $\id_M$ and some homotopy equivalence $\star M \to M$. Therefore, any $N$ homotopy equivalent to $M$ and with the same Kirby--Siebenmann invariant is in the same orbit as $\id_M$ under the action of the self homotopy equivalences of $M$.
\end{rem}

\bibliographystyle{amsalpha}
\bibliography{classification}

\providecommand{\bysame}{\leavevmode\hbox to3em{\hrulefill}\thinspace}
\providecommand{\MR}{\relax\ifhmode\unskip\space\fi MR }
\providecommand{\MRhref}[2]{%
  \href{http://www.ams.org/mathscinet-getitem?mr=#1}{#2}
}
\providecommand{\href}[2]{#2}
\begin{thebibliography}{{Wal}67}

\bibitem[CH90]{cochran-habegger}
T.~D. Cochran and N.~Habegger, \emph{On the homotopy theory of simply connected
  four manifolds}, Topology \textbf{29} (1990), no.~4, 419--440.

\bibitem[CH00]{CH}
A.~Cavicchioli and F.~Hegenbarth, \emph{{On the homotopy classification of
  4-manifolds having the fundamental group of an aspherical 4-manifold}},
  {Osaka J. Math.} \textbf{37} (2000), no.~4, 859--871.

\bibitem[FQ90]{Freedman-Quinn}
M.~H. Freedman and F.~Quinn, \emph{Topology of 4-manifolds}, Princeton
  Mathematical Series, vol.~39, Princeton University Press, Princeton, NJ,
  1990.

\bibitem[Fre82]{Freedman82}
M.~H. Freedman, \emph{The topology of four-dimensional manifolds}, J.
  Differential Geom. \textbf{17} (1982), no.~3, 357--453.

\bibitem[FT95]{FT}
M.~H. Freedman and P.~Teichner, \emph{{$4$}-manifold topology. {I}.
  {S}ubexponential groups}, Invent. Math. \textbf{122} (1995), no.~3, 509--529.

\bibitem[Hil06]{hillman}
J.~A. Hillman, \emph{{${\rm PD}_4$}-complexes with strongly minimal models},
  Topology Appl. \textbf{153} (2006), no.~14, 2413--2424.

\bibitem[HK93]{hambleton-kreck-93-III}
I.~Hambleton and M.~Kreck, \emph{{Cancellation, elliptic surfaces and the
  topology of certain four- manifolds.}}, {J. Reine Angew. Math.} \textbf{444}
  (1993), 79--100.

\bibitem[HKT09]{HKT}
I.~Hambleton, M.~Kreck, and P.~Teichner, \emph{Topological 4-manifolds with
  geometrically two-dimensional fundamental groups}, J. Topol. Anal. \textbf{1}
  (2009), no.~2, 123--151.

\bibitem[Kha12]{khan}
Q.~Khan, \emph{{Homotopy invariance of 4-manifold decompositions: connected
  sums}}, {Topology Appl.} \textbf{159} (2012), no.~16, 3432--3444.

\bibitem[KL09]{KL}
M.~Kreck and W.~L{\"{u}}ck, \emph{Topological rigidity for non-aspherical
  manifolds}, Pure Appl. Math. Q. \textbf{5} (2009), no.~3, Special Issue: In
  honor of Friedrich Hirzebruch. Part 2, 873--914.

\bibitem[KLPT17]{KLPT15}
D.~Kasprowski, M.~Land, M.~Powell, and P.~Teichner, \emph{Stable classification
  of 4-manifolds with 3-manifold fundamental groups}, J. Topol. \textbf{10}
  (2017), no.~3, 827--881.

\bibitem[KQ00]{KQ}
V.~S. Krushkal and F.~Quinn, \emph{Subexponential groups in 4-manifold
  topology}, Geom. Topol. \textbf{4} (2000), 407--430.

\bibitem[Kre99]{surgeryandduality}
M.~Kreck, \emph{Surgery and duality}, Ann. of Math. (2) \textbf{149} (1999),
  no.~3, 707--754.

\bibitem[KT01]{kirby-taylor}
R.~C. Kirby and L.~R. Taylor, \emph{A survey of 4-manifolds through the eyes of
  surgery}, Surveys on surgery theory, {V}ol. 2, Ann. of Math. Stud., vol. 149,
  Princeton Univ. Press, Princeton, NJ, 2001, pp.~387--421.

\bibitem[L{\"{u}}c19a]{Lueck-HH}
W.~L{\"{u}}ck, \emph{Assembly maps}, Handbook of homotopy theory; editor:
  Haynes Miller, CRC Press/Chapman and Hall (2019).

\bibitem[L{\"{u}}c19b]{LueckFJC}
\bysame, \emph{Isomorphism {C}onjectures in {K}- and {L}-theory}, availbale at
  https://www.him.uni-bonn.de/lueck/, ongoing book project, 2019.

\bibitem[Ran81]{Ranickiyellow}
A.~A. Ranicki, \emph{Exact sequences in the algebraic theory of surgery},
  Mathematical Notes, vol.~26, Princeton University Press, Princeton, N.J.;
  University of Tokyo Press, Tokyo, 1981.

\bibitem[Ran92a]{Ranicki-lower}
\bysame, \emph{Lower {$K$}- and {$L$}-theory}, London Mathematical Society
  Lecture Note Series, vol. 178, Cambridge University Press, Cambridge, 1992.

\bibitem[Ran92b]{Ranicki-blue-book}
Andrew~A. Ranicki, \emph{Algebraic {$L$}-theory and topological manifolds},
  Cambridge Tracts in Mathematics, vol. 102, Cambridge University Press,
  Cambridge, 1992. \MR{1211640 (94i:57051)}

\bibitem[Sha69]{Shaneson}
J.~L. Shaneson, \emph{Wall's surgery obstruction groups for {$G\times Z$}},
  Ann. of Math. (2) \textbf{90} (1969), 296--334.

\bibitem[{Wal}67]{Wall67}
C.~T.~C. {Wall}, \emph{{Poincar\'e complexes. I}}, {Ann. Math. (2)} \textbf{86}
  (1967), 213--245.

\end{thebibliography}

\end{document}